\documentclass[a4paper,11pt]{article}

\usepackage{amsmath}
\usepackage{amsfonts}
\usepackage{amssymb}
\usepackage{amsthm}
\usepackage{graphicx}
\usepackage{psfrag}
\usepackage{subfigure}
\newcommand{\R}{\mathbb{R}}

\newcommand{\bv}{\mathbf{v}}

\newcommand{\tbj}{\widetilde{\mathbf{J}}}
\newcommand{\bpsi}{\boldsymbol{\psi}}

\newcommand{\eps}{\ensuremath{\varepsilon}}

\newcommand{\duality}[2]{\left\langle #1, #2 \right\rangle}

\newcommand{\no}{\mathbf{n}}
\newcommand{\ve}{\mathbf{v}}
\newcommand{\we}{\mathbf{w}}

\newcommand{\Div}{\operatorname{div}}

\newtheorem{definition}{Definition}[section]

\newtheorem{lemma}[definition]{Lemma}
\newtheorem{theorem}[definition]{Theorem}

\numberwithin{equation}{section}  %equation-counter is set back to zero at the beginning of each section

\begin{document}
 %%% Neu:
\begin{titlepage}
\title{Convergence of a Nonlocal to a Local Diffuse Interface Model for Two-Phase Flow with Unmatched Densities}
\author{  Helmut Abels\footnote{Fakult\"at f\"ur Mathematik,  
Universit\"at Regensburg,
93040 Regensburg,
Germany, e-mail: {\sf helmut.abels@mathematik.uni-regensburg.de}}\ \ and 
Yutaka Terasawa\footnote{Graduate School of Mathematics, Nagoya University, Furocho Chikusaku, Nagoya, 464-8602, Japan, e-mail: {\sf yutaka@math.nagoya-u.ac.jp}}
}
\date{\today \\[3ex]\emph{Dedicated to Maurizio Grasselli on the occasion of his 60th birthday}}
\end{titlepage}
\maketitle
\begin{abstract}
We prove convergence of suitable subsequences of weak solutions of a  diffuse interface model for the two-phase flow of incompressible fluids with different densities with a nonlocal Cahn-Hilliard equation to weak solutions of the corresponding system with a standard ``local'' Cahn-Hilliard equation. The analysis is done in the case of a sufficiently smooth bounded domain with no-slip boundary condition for the velocity and Neumann boundary conditions for the Cahn-Hilliard equation. The proof is based on the corresponding result in the case of a single Cahn-Hilliard equation and compactness arguments used in the proof of existence of weak solutions for the diffuse interface model. 
\end{abstract}

\noindent{\bf AMS-Classification:} 
Primary: 76T99; %% Two-Phase flows: Others
Secondary: 35Q30, %% Stokes and Navier-Stokes eq.
35Q35, %% Other equations arising in fluid mechanics
76D03, %% Incompressible viscous fluids: Existence, uniqueness, and regularity theory 
76D05, %% Incompressible viscous fluids: Navier-Stokes equations
76D27, %% Incompressible viscous fluids: Other free-boundary flows; Hele-Shaw flows
76D45 %% Incompressible viscous fluids: Capillarity (surface tension)

%%%%%%%%%%%%%%%%%%%%%%%%%%%%%%%%%%%%%%%%%%%%%%%%%%%%%%%%%%%%%%%%%%%%%%%%%%%%%%%%%%%%%%%%%%%%%%%%
%%%%%%%%%%%%%%%%%%%%%%%%%%%%%%%%%%%%%%%%%%%%%%%%%%%%%%%%%%%%%%%%%%%%%%%%%%%%%%%%%%%%%%%%%%%%%%%%
%%%%%%%%%%%%%%%%%%%%%%%%%%%%%%%%%%%%%%%%%%%%%%%%%%%%%%%%%%%%%%%%%%%%%%%%%%%%%%%%%%%%%%%%%%%%%%%%

\section{Introduction} \label{intro}

In this paper, we consider the convergence of a non-local diffuse interface model for the two-phase flows of two incompressible fluids with unmatched densities to the corresponding ``local'' system. More precisely, we consider the non-local Navier-Stokes/Cahn-Hilliard system
\begin{alignat}{2}
 \partial_t (\rho_\eps \mathbf{v}_\eps) + \operatorname{div} ( \bv \otimes(\rho \bv_\eps + \tbj_\eps)) - \operatorname{div} (2 \nu(\varphi_\eps) D \bv_\eps)
  + \nabla p_\eps  
  & = \mu_\eps\nabla \varphi_\eps  &\quad & \text{in } Q_T ,  \label{eq:1} 
\\
 \operatorname{div} \, \bv_\eps &= 0&& \text{in } Q_T,  \label{eq:2} 
\\
 \partial_t \varphi_\eps + \bv_\eps \cdot \nabla \varphi_\eps &= \operatorname{div}\left(m(\varphi_\eps) \nabla \mu_\eps \right)&& \text{in } Q_T, \label{eq:3} 
\\
  F'(\varphi_\eps)  +a_\eps(x)\varphi_\eps - J_\eps \ast \varphi_\eps &= \mu_\eps && \text{in } Q_T, \label{eq:4} 
\end{alignat}
where $\rho_\eps=\rho(\varphi_\eps):= \frac{\tilde{\rho}_1+\tilde{\rho}_2}2+ \frac{\tilde{\rho}_2-\tilde{\rho}_1}2\varphi_\eps $ is the density of the mixture of the two fluids, $\tilde{\rho}_1, \tilde{\rho}_2>0$ are the specific constant mass densities of the unmixed fluids,
$$
\tbj_\eps = -\frac{\tilde{\rho}_2 - \tilde{\rho}_1}{2} m(\varphi_\eps) \nabla \mu
$$ is a relative mass flux, 
$Q_T=\Omega\times(0,T)$, where $T\in (0,\infty)$ is arbitrary. We assume that $\Omega \subset \mathbb{R}^d$, $d=2,3$, is a bounded domain with $C^2$-boundary.
Here $\ve_\eps\colon Q_T \to \R^d$ is the (mean) velocity of the fluid mixture, $p_\eps\colon Q_T\to \R$ is its pressure, $\varphi_\eps \colon Q_T \to \R$ is the difference of volume fractions of the fluids, and $\mu_\eps\colon Q_T \to \R$ is the chemical potential related to $\varphi_\eps$, which is the first variation of the (non-local) free energy
\begin{equation*}
  E_\eps (\varphi) = \frac14 \int_\Omega \int_\Omega J_\eps (x-y) |\varphi(x)-\varphi(y)|^2\, dx\, dy +\int_\Omega F(\varphi(x))\, dx.
\end{equation*}
Moreover, $J_{\eps}$ is a nonnegative function on $\R^d$, $F$ is a homogeneous free energy density and
\begin{equation*}
J_\eps \ast \varphi(x) := \int_{\Omega} J_\eps(x-y) \varphi(y)\,dy,\quad a_\eps(x) :=  \int_{\Omega} J_\eps(x-y) \,dy \qquad\text{for all } x \in \Omega.
\end{equation*} 
More precisely, we assume that $J_{\eps}(x) = \frac{\eta_{\eps}(|x|)}{|x|^{2}} $ for all $x\in\R^d$ and $ J_{\eps} \in W^{1, 1} (\mathbb{R}^d) $ for $ \eps > 0$ and $(\eta_{\eps})_{\eps >0}$ is a family of molifiers with the following properties:
\begin{alignat}{2}
&\eta_{\eps} \colon \mathbb{R} \longrightarrow [0, +\infty),~~~\eta_{\eps} \in L^{1}_{loc}(\mathbb{R}),&~~~\eta_{\eps}(r) = \eta_{\eps}(-r)~~~~\forall r \in \mathbb{R},\eps>0; \nonumber\\
&\int_{0}^{+\infty} \eta_{\eps}(r)r^{d-1}\,dr = \frac{2}{C_d}~~~~\forall \eps >0; && \nonumber \\
&\lim_{\eps \rightarrow 0+} \int_{\delta}^{+\infty} \rho_{\eps}(r)r^{d-1}\,dr = 0~~~\forall \delta > 0,&& \nonumber 
\end{alignat} 
where $C_d:= \int_{S^{d-1}} |e_1 \cdot \sigma|^2 d \mathcal{H}^{d-1}(\sigma)$. Moreover, we assume that
$F\colon [-1,1]\to \R$ is given by
\begin{equation*} 
F(s) = \frac{\theta}{2} ((1+s) \log(1+s) + (1-s) \log (1-s)) - \frac{\theta_c}{2}s^2 \qquad \text{for all }s\in [-1,1]
\end{equation*}  
for some $0 < \theta_c < \theta$ for simplicity. But every $F$ satisfying the assumptions in \cite{DavoliEtAlW11} and \cite{FrigeriNonlocalAGG}  can be treated as well. Finally, $\nu \colon [-1,1]\to (0,\infty)$ and $m \colon [-1,1]\to (0,\infty)$ are viscosity and mobility coefficients, which are assumed to be sufficiently smooth. %%% Specify??

We complement the system \eqref{eq:1}-\eqref{eq:4} with
the following boundary and initial conditions:
\begin{alignat}{3}
\mathbf{v}_{\eps}|_{\partial\Omega}&=0,& \quad \left.\frac{\partial \mu_{\eps}}{\partial \no}\right|_{\partial\Omega}&=0 &\quad&\text{on }\partial \Omega\times (0,T), \label{bc:1} \\
\mathbf{v}_{\eps}|_{t=0} &= \mathbf{v}_{0,\eps},& \varphi_{\eps}|_{t=0}&=\varphi_{0,\eps} &\quad&\text{in }\Omega. \label{ic:1}
\end{alignat}
The system \eqref{eq:1}-\eqref{eq:4} is a variant of the following diffuse interface model for the two-phase flows of two incompressible fluids with unmatched densities, which was derived in \cite{AbelsGarckeGruen2}:
\begin{alignat}{2}
 \partial_t (\rho \mathbf{v}) + \operatorname{div} ( \bv \otimes(\rho \bv + \tbj)) - \operatorname{div} (2 \nu(\varphi) D \bv)
  + \nabla p  
  & = \mu \nabla \varphi&\quad & \text{in } Q_T ,  \label{eq:AGG1} 
\\
 \operatorname{div} \, \bv &= 0&& \text{in } Q_T,  \label{eq:AGG2} 
\\
 \partial_t \varphi + \bv \cdot \nabla \varphi &= \operatorname{div}\left(m(\varphi) \nabla \mu \right)&& \text{in } Q_T, \label{eq:AGG3} 
\\
 \mu&=  F'(\varphi)  -\Delta \varphi && \text{in } Q_T, \label{eq:AGG4} 
\end{alignat}
where $\rho=\rho(\varphi):= \frac{\tilde{\rho}_1+\tilde{\rho}_2}2+ \frac{\tilde{\rho}_2-\tilde{\rho}_1}2\varphi $ is the density of the mixture of the two fluids and
$
\tbj = -\frac{\tilde{\rho}_2 - \tilde{\rho}_1}{2} m(\varphi) \nabla \mu
$ is a relative mass flux as before. This system is complemented by the boundary and initial conditions
\begin{alignat}{3}
\mathbf{v}|_{\partial\Omega}&=0,& \quad \left.\frac{\partial \mu}{\partial \no}\right|_{\partial\Omega}&=\left.\frac{\partial \varphi}{\partial \no}\right|_{\partial\Omega}=0 &\quad&\text{on }\partial \Omega\times (0,T), \label{bc:2} \\
\mathbf{v}|_{t=0} &= \mathbf{v}_0,& \varphi|_{t=0}&=\varphi_0 &\quad&\text{in }\Omega. \label{ic:2}
\end{alignat}
We note that \eqref{eq:1}-\eqref{eq:4} is obtained by the latter system by replacing the standard ``local'' Cahn-Hilliard equation \eqref{eq:AGG3}-\eqref{eq:AGG4} (with an additional convection term $\bv\cdot \nabla \varphi$) by its non-local variant \eqref{eq:3}-\eqref{eq:4}. Moreover, note that in \eqref{bc:2} an additional Neumann boundary condition for $\varphi$ is present, which is not posed for the non-local system. 

It is the goal of the present contribution to show convergence of weak solutions of \eqref{eq:1}-\eqref{eq:4} together with \eqref{bc:1}-\eqref{ic:1} to a weak solution of \eqref{eq:AGG1}-\eqref{eq:AGG4} together with \eqref{bc:2}-\eqref{ic:2} for a suitable subsequence and under suitable conditions on the initial values. Existence of weak solutions of \eqref{eq:AGG1}-\eqref{ic:2} was first proven by A., Depner, and Garcke in \cite{AbelsDepnerGarcke}. Existence of strong solutions for small times was proved by Weber~\cite{ThesisWeber}, cf.\ also A.\ and Weber~\cite{AbelsWeberAGG}. A result on well-posedness of this system in two-space dimensions and further references can be found in the recent contribution by Giorgini~\cite{GiorginWellposednessAGG}.  The existence of weak solutions to the non-local model \eqref{eq:1}-\eqref{eq:4} together with \eqref{bc:1}-\eqref{ic:1} was proved by Frigeri~\cite{FrigeriNonlocalAGG} for suitable integrable kernels $J_\eps$ and by the authors in \cite{AbelsTerasawaNonlocalAGG} for singular kernels. We refer to Frigeri~\cite{FrigeriNonlocalDegAGG} for a recent overview of the literature for these non-local models, to Gal, Grasselli, and Wu~\cite{GalGrasselliWu} for a recent result on the (local) Navier-Stokes/Cahn-Hilliard system with different densities and further references, and to Frigeri, Gal, and Grasselli~\cite{FigeriGalGrasselliNonlocalCH} for a recent result on the nonlocal Cahn-Hilliard equation with singular potentials and degenerate mobility and further references.

Convergence of solutions of the nonlocal Cahn-Hilliard equation, i.e., \eqref{eq:3}-\eqref{eq:4} with $\ve_\eps\equiv 0$, to the local Cahn-Hilliard equation, i.e., \eqref{eq:AGG3}-\eqref{eq:AGG4} with $\ve\equiv 0$, was proved by Melchionna et al.~\cite{MelchionnaEtAlNonlocalLocalCH} in the case of periodic boundary conditions and a regular free energy density $F$, by Davoli et al.\ \cite{DavoliEtAlNonlocalLocalCHPeriodic} in the case of periodic boundary conditions and singular free energy densities, by Davoli et al.\ \cite{DavoliEtAlNonlocalLocalCHViscosity} in the case of Neumann boundary conditions with an additional viscosity term in the nonlocal Cahn-Hilliard equation and in \cite{DavoliEtAlW11} in the case of Neumann boundary conditions and $W^{1,1}$-kernels. We note that these results are based on the results of Ponce~\cite{PonceNonlocalPoincare,PonceNonlocalLocalConvergence}, which in particular yield $\Gamma$-convergence of the non-local free energy of $E_\eps$ to the corresponding local free energy
\begin{equation*}
 E(\varphi):= \int_\Omega \frac{|\nabla \varphi|^2}2\, dx + \int_\Omega F(\varphi)\, dx.
\end{equation*}
We refer to \cite{DavoliEtAlW11} for further references.

In this contribution we combine the arguments from \cite{DavoliEtAlW11} for the convergence of the nonlocal to the local Cahn-Hilliard equation and \cite{AbelsDepnerGarcke} for the existence of weak solutions to the limit system \eqref{eq:AGG1}-\eqref{ic:2} to show convergence of weak solutions of \eqref{eq:1}-\eqref{eq:4} together with \eqref{bc:1}-\eqref{ic:1} to a weak solution of \eqref{eq:AGG1}-\eqref{eq:AGG4} together with \eqref{bc:2}-\eqref{ic:2} for a suitable subsequence. The structure of this contribution is as follows: In Section~\ref{prelimi} we recall some preliminary results and basic definition. Then the main result is proved in Section~\ref{main}.

\section{Preliminaries} \label{prelimi}

In the following $D\ve= \frac12 (\nabla \ve + \nabla \ve^T)$ denotes the symmetric part of the gradient of a vector field $\ve$. The tensor product $a \otimes b$ of the vectors $a$ and $b$ is $ (a \otimes b)_{i, j} = a_i b_j $ for $i, j = 1, \cdots, d$. For a normed space $X$ we denote by
\begin{equation*}
  \duality{x'}{x}_X := x'(x)\qquad \text{for all }x'\in X',x\in X
\end{equation*}
its duality product. $C([0,T];X)$ denotes the space of all strongly continuous $f\colon [0,T]\to X$ equipped with the supremum-norm.
$BC_w([0,T];X)$ denotes the spaces of all bounded and weakly continuous $f\colon [0,T]\to X$ equipped with the supremum-norm. 
If $M \subset \mathbb{R}^d$ is measurable, $L^q(M)$ denotes the usual Lebesgue space and $\| \cdot \|_q$ its norm. Moreover, $L^q(M; X)$ denotes the set of all strongly measurable $q$-integrable functions/essentially bounded functions, where $X$ is a Banach space. If $M=(a, b)$, we write for simplicity $L^q(a, b)$ and $L^q(a, b; X)$.

Let $ \Omega \subset \mathbb{R}^d$ be a domain. Then $W^m_q(\Omega)$, $m \in \mathbb{N}_0$, $1 \leq q \leq \infty$, denotes the usual $L^q$-Sobolev space,
$W^{m}_{q, 0}(\Omega)$ the closure of $C^{\infty}_0(\Omega)$ in $W^m_q(\Omega)$ and $W^{-m}_q(\Omega)=(W^m_{q', 0})'$. $H^s(\Omega)$, $s \in \mathbb{R}$, denotes the usual $L^2$-Bessel potential spaces and $H^s_0(\Omega)$ the closure of $C^{\infty}_0(\Omega)$ in $H^s(\Omega)$ when $s$ is positive.

%We set $H:=L^2(\Omega)$, $V:=H^1(\Omega)$ and we denote by $\| \cdot \|$ and
%$( \cdot, \cdot)$ the standard norm and the scalar product in $H$ as well as  in $L^2(\Omega)$ and $L^2(\Omega)^{d \times d}$.
Furthermore,
$$\mathcal{D}(A) = H^2(\Omega)^d \cap H^1_0(\Omega)^d\cap L^2_\sigma(\Omega)$$
denotes the domain of the Stokes operator on $L^2_\sigma(\Omega):= \overline{\{\bpsi \in C^\infty_0 (\Omega)^d: \Div \bpsi=0\}}^{L^2(\Omega)}$.

We consider weak solutions of \eqref{eq:1}-\eqref{eq:4} together with \eqref{bc:1}-\eqref{ic:1} in the following sense:
\begin{definition} \label{def-weak}
  Let $\ve_0 \in L^2_\sigma(\Omega)$, $\varphi_0 \in L^\infty(\Omega)$ with $|\varphi_0|\leq 1$ almost everywhere and $T\in (0,\infty)$, $\eps>0$ be given. Then $(\ve,\varphi,\mu)$ is a weak solution of \eqref{eq:1}-\eqref{ic:1} if
  \begin{align*}
    \ve &\in BC_w([0,T]; L^2_\sigma(\Omega))\cap L^2(0,T; H^1_0(\Omega)^d),\\
    \varphi &\in L^\infty(0,T;L^2(\Omega))\cap L^2(0,T;H^1(\Omega)),\\
    \mu= a_\eps \varphi-J_\eps\ast \varphi +F'(\varphi) &\in L^2(0,T; H^1(\Omega)),\\
    \partial_t (\rho \ve)&\in L^{4/3} (0,T;\mathcal{D}(A)'), \quad \partial_t\varphi \in L^2(0,T;(H^1(\Omega))'),
  \end{align*}
 $|\varphi(x,t)|<1$ almost everywhere in $Q_T$, $\ve|_{t=0}= \ve_0$, $\varphi|_{t=0}= \varphi_0$ and the following holds true:
  \begin{enumerate}
  \item For every $\psi \in H^1(\Omega)$ and $\bpsi\in \mathcal{D}(A)$ and almost every $t\in (0,T)$ we have
    \begin{align*}
      \duality{\partial_t (\rho\ve)(t)}{\bpsi}_{\mathcal{D}(A)} &- \int_\Omega ( (\ve+ \tbj) \otimes \rho\ve : D\bpsi \,dx + \int_\Omega 2\nu(\varphi)D\ve:D\bpsi\, dx = -\int_\Omega \varphi \nabla \mu\cdot \bpsi \,dx,\\
      \duality{\partial_t \varphi(t)}{\psi}_{H^1(\Omega)} &+ \int_\Omega m(\varphi) \nabla \mu\cdot \nabla \psi = \int_\Omega \ve \varphi \cdot \nabla \psi\,dx,
    \end{align*}
    where $\tbj= -\tfrac{\tilde{\rho}_2 - \tilde{\rho}_1}{2}m(\varphi) \nabla \mu$.
  \item The energy inequality
    \begin{equation}
      \label{eq:Energy}
      \mathcal{E}_\eps(\ve(t),\varphi(t))+ \int_0^t\int_\Omega (2\nu(\varphi)|D\ve|^2+ m(\varphi)|\nabla \mu|^2) \,dx\, d\tau \leq \mathcal{E}_\eps(\ve(0),\varphi(0))
    \end{equation}
    holds true for all $t \in [0,T]$, where
    \begin{align*}
            \mathcal{E}_\eps(\ve,\varphi)&:= \frac12\int_\Omega \rho(\varphi)|\ve|^2\, dx + E_\eps(\varphi),\\
      E_\eps(\varphi) &:= \frac14 \int_\Omega \int_\Omega J_\eps(x-y) (\varphi(x)-\varphi(y))^2 \,dx\, dy + \int_\Omega F(\varphi(x))\, dx.
    \end{align*}
  \end{enumerate}
\end{definition}
Existence of weak solutions  for any $\ve_0 \in L^2_\sigma(\Omega)$, $\varphi_0 \in L^\infty(\Omega)$ with $|\varphi_0|\leq 1$ almost everywhere and $T\in (0,\infty)$, $\eps>0$ follows from \cite[Theorem~1]{FrigeriNonlocalAGG}.

For the following we denote by
\begin{alignat*}{2}
  E_\eps^0(\varphi) &:= \frac14 \int_\Omega \int_\Omega J_\eps(x-y) (\varphi(x)-\varphi(y))^2 \,dx\, dy&\quad &\text{for } \varphi \in L^2(\Omega),\\
  E^0 (\varphi) &:= \frac12\int_\Omega |\nabla \varphi(x)|^2\,dx&\quad &\text{for } \varphi \in H^1(\Omega)
\end{alignat*}
the first parts of the free energies in the nonlocal and local case. We note that
\begin{equation*}
  E_\eps^0(\varphi) \leq E_\eps(\varphi) + C,\qquad E^0(\varphi) \leq E(\varphi)+C
\end{equation*}
for some $C>0$ independent of $\eps\in (0,1)$ since $F\colon [-1,1]\to \R$ is bounded below.
The following two lemmas will be important to obtain compactness as $\eps \to 0$:
\begin{lemma}\label{lem:3.3}
  For every $\varphi, \zeta\in H^1(\Omega)$ it holds that
  \begin{align*}
    \lim_{\eps\to 0} E^0_\eps(\varphi) &= E^0(\varphi),\\
    \lim_{\eps \to 0} \int_\Omega (a_\eps \varphi-J_\eps \ast \varphi)(x)\zeta (x) \,dx &= \int_\Omega \nabla \varphi(x)\cdot \nabla \zeta(x)\, dx.
  \end{align*}
  Moreover, for every sequence $(\varphi_\eps)_{\eps>0}\subseteq L^2(\Omega)$ and $\varphi \in L^2(\Omega)$ it holds that
  \begin{alignat*}{2}
    \sup_{\eps >0} E^0_\eps (\varphi) <+\infty & \quad \Rightarrow &\quad& (\varphi_\eps)_{\eps>0} \text{ is relatively compact in } L^2(\Omega),\\
    \varphi_\eps \to_{\eps\to 0} \varphi\quad \text{in } L^2(\Omega) &\quad \Rightarrow & \quad& E^0(\varphi)\leq \liminf_{\eps\to 0} E^0_\eps (\varphi_\eps).
  \end{alignat*}
\end{lemma}
\noindent We refer to \cite[Lemma~3.3]{DavoliEtAlW11} for the proof of this lemma.

\begin{lemma}\label{lem:3.4}
For any $\delta>0$, there exists some $C_{\delta}>0$ and $\eps_{\delta}>0$ such that for any $(\varphi_{\eps})_{\eps>0} \subset L^2(\Omega)$
\begin{equation}
\| \varphi_{\eps_1} - \varphi_{\eps_2} \|^2_{L^2(\Omega)} \leq \delta (E^0_{\eps_1}(\varphi_{\eps_1}) + E^0_{\eps_2}(\varphi_{\eps_2}))
+ C_{\delta} \| \varphi_{\eps_1} - \varphi_{\eps_2} \|^2_{(H^1(\Omega))'} 
\end{equation}  
holds for any $\eps_1, \eps_2 \in (0, \eps_{\delta})$. 
\end{lemma}
\noindent The lemma is proved in \cite[Lemma~3.4]{DavoliEtAlW11}.

\section{Main Result}\label{main}

\begin{theorem}\label{main-thm}
 For any $\eps \in (0,1)$ let $\mathbf{v}_{0,\eps} \in L_\sigma^2(\Omega)$ and $\varphi_{0,\eps} \in L^\infty(\Omega)$ with $|\varphi_{0,\eps}(x)| \leq 1$ almost everywhere, $\frac1{|\Omega|}\int_\Omega \varphi_{0,\eps}(x)\, dx = m_\Omega$ for all $\eps\in (0,1)$ and some $m_\Omega\in (-1,1)$. Moreover, we assume that there are $\ve_0\in L^2_\sigma(\Omega)$ and $\varphi_0 \in H^1(\Omega)$ such that $\ve_{0,\eps}\to_{\eps\to 0} \ve_0$ in $L^2_\sigma(\Omega)$, $\varphi_{0,\eps}\to_{\eps \to 0} \varphi_0$ in $L^2(\Omega)$, and
  \begin{equation*}
    \mathcal{E}_\eps (\ve_{0,\eps}, \varphi_{0,\eps})\to_{\eps \to 0} \mathcal{E}(\ve_0,\varphi_0),
  \end{equation*}
   where
    \begin{align*}
      \mathcal{E}(\ve,\varphi)&:= \frac12\int_\Omega \rho(\varphi)|\ve|^2\, dx + E(\varphi),\quad
      E(\varphi) := \frac12 \int_\Omega |\nabla \varphi(x)|^2 \,dx + \int_\Omega F(\varphi(x))\, dx.
    \end{align*}
If $\mathbf{v}_{\eps}$, $\varphi_{\eps}$ and $\mu_{\eps}$ are weak solutions of \eqref{eq:1}-\eqref{ic:1} with initial values $(\ve_{0,\eps}, \varphi_{0,\eps})$, then  
\begin{alignat}{2}\label{eq:Conv1}
\mathbf{v}_{\eps} &\rightharpoonup \mathbf{v}&\qquad &\text{weakly-}\ast \text{ in }L^{\infty}(0, T;L^2_{\sigma}(\Omega)), \\\label{eq:Conv2}
\mathbf{v}_{\eps} &\rightharpoonup \mathbf{v}&&\text{weakly in }L^2(0, T; H^1_0 (\Omega)^d), \\\label{eq:Conv3}
%\varphi_{\eps} &\rightharpoonup \varphi &&\text{weakly-}\ast \text{ in }L^{\infty}(0, T; L^2(\Omega)),\\
\mu_{\eps} &\rightharpoonup \mu && \text{weakly in }L^2(0, T;V),\\\label{eq:Conv4} 
\mathbf{v}_{\eps} &\rightarrow \mathbf{v} && \text{strongly in }L^2(0, T; L^2(\Omega)) \text{ and almost everywhere}, \\\label{eq:Conv5}
\varphi_{\eps} &\rightarrow \varphi && \text{strongly in  }C([0,T];L^2(\Omega)) \text{ and almost everywhere}
\end{alignat}
for a suitable subsequence $\eps=\eps_k\to_{k\to\infty} 0$,
where  $(\mathbf{v},\varphi,\mu)$ is a weak solution \eqref{eq:AGG1}-\eqref{ic:2} in the sense that
\begin{align*}
  \ve &\in BC_w([0,T];L^2_\sigma(\Omega))\cap L^2(0,T;H^1_0(\Omega)^d),\\
  \varphi & \in C([0,T]; L^2(\Omega))\cap BC_w([0,T]; H^1(\Omega))\cap L^2(0,T;H^2(\Omega)), F'(\varphi) \in L^2(0,T;L^2(\Omega)),\\
  \mu & \in L^2(0,T; H^1(\Omega)),
\end{align*}
$|\varphi(x,t)|<1$ almost everywhere in $Q$, and the following holds true:
\begin{enumerate}
\item $\mu = -\Delta \varphi + F'(\varphi)$ almost everywhere in $Q_T$.
  \item For every $\psi \in H^1_0(\Omega)^d\cap L^2_\sigma(\Omega)$ and $\bpsi\in \mathcal{D}(A)$ and almost every $t\in (0,T)$ we have
    \begin{align*}
      \duality{\partial_t (\rho\ve)}{\bpsi}_{\mathcal{D}(A)} &- \int_\Omega ( (\ve+ \tbj) \otimes \rho\ve : D\bpsi \,dx + \int_\Omega 2\nu(\varphi)D\ve:D\bpsi\, dx = -\int_\Omega \varphi \nabla \mu\cdot \bpsi \,dx,\\
      \duality{\partial_t \varphi}{\psi}_{H^1(\Omega)} &+ \int_\Omega m(\varphi) \nabla \mu\cdot \nabla \psi = \int_\Omega \ve \varphi \cdot \nabla \psi\,dx,
    \end{align*}
    where $\tbj= -\tfrac{\tilde{\rho}_2 - \tilde{\rho}_1}{2}m(\varphi) \nabla \mu$.
  \item The energy inequality
    \begin{equation}
      \label{eq:Energy2}
      \mathcal{E}(\ve(t),\varphi(t))+ \int_0^t\int_\Omega (2\nu(\varphi)|D\ve|^2+ m(\varphi)|\nabla \mu|^2 \,dx\, d\tau \leq \mathcal{E}(\ve_0,\varphi_0)
    \end{equation}
    holds true for all $t\in [0,T]$.
  \end{enumerate}
%where  $(\mathbf{v},\varphi,\mu)$ is a weak solution \eqref{eq:AGG1}-\eqref{ic:2}, which is defined in Definition \ref{def-weak}.
\end{theorem}

\begin{proof}
From the energy inequality \eqref{eq:Energy}, we see that $(\mathbf{v}_{\eps})_{\eps \in (0,1)}$ is bounded in $L^{\infty}(0, T;L^2_{\sigma}(\Omega))$ and $L^2(0, T; H^1_{0}(\Omega)^d)$. Hence one can find a subsequence such that \eqref{eq:Conv1} and \eqref{eq:Conv2} hold. Moreover, since $|\varphi_\eps(x,t)|<1$ almost everywhere, $(\varphi_{\eps})_{\eps \in (0,1)}$ is obviously bounded in $L^{\infty}(0, T; L^2(\Omega))$. We also see from the energy inequality that $(\nabla \mu_{\eps})_{\eps\in (0,1)}$ is bounded in $L^2(0, T;L^2(\Omega))$. To see that 
$(\mu_{\eps})_{\eps (0,1)}$ is bounded in $L^2(0, T;H^1(\Omega))$, we know from the Poincar\'e-Wirtinger inequality that it is enough to show that
$(\mu)_{\Omega}:=\frac{1}{|\Omega|} \int_{\Omega} \mu\,d x \in L^2(0, T)$. 
The argument below for showing this are an adaptation of the arguements in Section~4.1 of \cite{DavoliEtAlW11}. We include it for the reader's convenience.

For the following we define
$$
\mathcal{N}(\varphi_\eps(t))\colon (H^1_{(0)}(\Omega))'\to H^1_{(0)}(\Omega):= \left\{u\in H^1(\Omega): \int_\Omega u\,d x=0\right\}\colon f\mapsto u, 
$$
where $u\in H^1_{(0)}(\Omega)$ is the solution of
\begin{equation*}
  \int_\Omega m(\varphi_\eps (t))\nabla u \cdot \nabla \psi \, dx = \duality{f}{\psi}\qquad \text{for all }\psi\in H^1_{(0)}(\Omega).
\end{equation*}
Since $m$ is strictly bounded below (independent of $\varphi_\eps(t)$), there is some constant $C$, independent of $\varphi_\eps (t)$, such that
\begin{equation*}
  \|\mathcal{N}(\varphi_\eps(t))f\|_{H^1(\Omega)}\leq C\|f\|_{(H^1_{(0)}(\Omega))'}\qquad \text{for all }f\in (H^1_{(0)}(\Omega))'. 
\end{equation*}
 Then testing \eqref{eq:3} by $\mathcal{N}(\varphi_\eps(t))(\varphi_{\eps}(t) - m_{\Omega})$ (in the weak sense), \eqref{eq:4} with $\varphi_{\eps}(t)-m_\Omega$ and taking the sum yields
\begin{align}
& \duality{\partial_t \varphi_{\eps}(t)}{\mathcal{N}(\varphi_\eps(t))(\varphi_{\eps}(t) - m_\Omega)}_{H^1_{(0)}(\Omega)} + 2E^0_{\eps}(\varphi_{\eps}(t)) 
+ \int_{\Omega} F'_0(\varphi_{\eps}(x, t))(\varphi_{\eps}(x, t) - m_{\Omega})\, dx \nonumber \\
& = \int_{\Omega} \varphi_{\eps}(x, t) \mathbf{v}_{\eps}(x, t) \cdot \nabla \mathcal{N}(\varphi_\eps (t))(\varphi_{\eps}(x, t) - m_{\Omega})\, dx - \int_{\Omega} \theta_c \varphi_{\eps}(x, t) (\varphi_{\eps}(x, t) - m_{\Omega})\, dx, \label{testing eq:1}  
\end{align}
where
$$
F_0(s):= F(s)+\theta_c \frac{s^2}2\qquad \text{for }s\in [-1,1]
$$
is the ``convex part'' of $F$. Using the weak form of \eqref{eq:3} it is easy to see that $\partial_t \varphi_{\eps} $ is bounded in $L^2(0, T;(H^1(\Omega))')$ since $\nabla \mu_\eps$ and $\varphi_\eps \ve_\eps$ are bounded in $L^2(0,T;L^2(\Omega)^d)$. 
Using this and the energy inequality, we observe that the first term in the left-hand side of \eqref{testing eq:1} is bounded in $L^2(0, T)$ independently of $\eps$ and the second is non-negative. Using the properties of $\mathcal{N}$ and the H\"older inequality, the right-hand side of \eqref{testing eq:1} can be estimated from above by a constant multiple of
\begin{align}
& \| \varphi_{\eps} \|_{L^{\infty}(\Omega)} \| \mathbf{v}_{\eps} \|_{L^2(\Omega)} \| \varphi_{\eps} - m_{\Omega} \|_{(H^1_{(0)}(\Omega))'} 
+ \theta_c \| \varphi_{\eps} \|_{L^2(\Omega)} \| \varphi_{\eps} - m_{\Omega} \|_{L^2(\Omega)}.  
\end{align}
Hence they are bounded in $L^{\infty}(0, T)$. Moreover, using the estimate
in the last line of p.\ 462 in \cite{AbelsDepnerGarcke}, %which is originally due to Kenmochi et al.\footnote{Add reference!},
there exist constants $c_1$ and $c_2$ such that 
\begin{align}
& \int_{\Omega} F'_0(\varphi_{\eps}(x, t))(\varphi_{\eps}(x, t) - m_{\Omega}) dx \geq c_1 \| F'_0(\varphi_{\eps}) \|_{L^1(\Omega)} - c_2.
\end{align}
% We add a remark that $ \eps $ which appeared for the derivation of the above estimate in the paper by Abels et al. has nothing to do with $\eps$ here.
Combining these estimates and \eqref{testing eq:1}, we have that $F'_0(\varphi_\eps)$ is bounded in $L^2(0,T;L^1(\Omega))$. Moreover, integrating $\mu_\eps = a_\eps \varphi_\eps - J_\eps \ast \varphi_\eps + F'(\varphi_\eps)$ in $\Omega$, yields that
\begin{equation*}
 (\mu_{\eps})_{\Omega} = \frac1{|\Omega|}\int_\Omega (F'_0(\varphi_\eps)-\theta_c \varphi_\eps)\, dx
\end{equation*}
is bounded in $L^2(0,T)$. Hence $(\mu_\eps)_{\eps \in (0,1)}$ is bounded in $L^2(0,T;H^1(\Omega))$ and we can choose a subsequence such that \eqref{eq:Conv3} holds.

Next we show \eqref{eq:Conv5} for a suitable subsequence.
As seen before $ (\partial_t \varphi_{\eps})_{\eps \in (0,1)} \subseteq L^2(0, T;(H^1(\Omega))') $ is bounded. Furthermore $ (\varphi_{\eps})_{\eps\in (0,1)} $ is bounded in $L^{\infty}(0, T;L^2(\Omega))$ since 
$ |\varphi_\eps(x, t)|<1$ almost everywhere in $Q_T$. Since $L^2(\Omega)$ is compactly embedded in $(H^1(\Omega))'$,  we have $ \varphi_{\eps} \rightarrow \varphi $ in $C([0, T]; (H^1(\Omega))')$ for a suitable subsequence by the Aubin-Lions lemma. Using Lemma~\ref{lem:3.4} and the bounds on the energies, we have $ \varphi_{\eps} \rightarrow \varphi $ in $C([0, T]; L^2(\Omega))$ and almost everywhere for a suitable subsequence. 
Since the function $\rho(\varphi_\eps)$ is bounded and depends continuously on $\varphi_\eps$, using Lebesgue's dominated convergence theorem, we have
$ \rho(\varphi_{\eps}) \rightarrow \rho(\varphi)$ strongly in $L^{q}(\Omega) $ for any $1 \leq q < \infty.$ Using the energy inequality, $\mathbf{v}_{\eps}$ is uniformly bounded in $L^{\infty}(0, T; L^2_\sigma(\Omega)) \cap L^2(0, T;L^6(\Omega)^d)$ and hence also in $L^{\frac{10}{3}}(Q_T)^d$. Thus
$\mathbf{v}_{\eps} \rightharpoonup \mathbf{v}$ weakly in $L^{\frac{10}{3}}(Q_T)^d$. Combining these convergence result, we derive $\rho(\varphi_{\eps}) \mathbf{v}_{\eps} \rightharpoonup \rho(\varphi)\mathbf{v}$ weakly in $L^{\frac{10}{3}-\gamma}(Q_T)^d$ for any $0<\gamma<\frac{10}{3}$ and a suitable subsequence. 
This implies
\begin{equation*}
  \rho(\varphi_{\eps}) \mathbf{v}_{\eps} \rightharpoonup \rho(\varphi) \mathbf{v}\qquad \text{ weakly in }L^2(0, T;L^2(\Omega)^d).  
\end{equation*}
 Since the Helmholtz projection $\mathbb{P}_{\sigma}$ is weakly continuous in $L^2(0, T; L^2(\Omega)^d)$, we obtain
 \begin{equation*}
   \mathbb{P}_{\sigma}(\rho(\varphi_{\eps})\mathbf{v}_{\eps}) \rightharpoonup \mathbb{P}_{\sigma}(\rho(\varphi)\mathbf{v})\qquad \text{ weakly in }L^2(0, T; L^2(\Omega)^d).  
 \end{equation*}
%From the energy inequality, we have $ \rho_{\eps} \mathbf{v}_{\eps} $   is bounded in $L^2(0, T; L^2(\Omega)).$ 
Using the weak form of \eqref{eq:1}, we have
\begin{align}\nonumber
  \duality{\partial_t (\mathbb{P}_{\sigma}(\rho_{\eps} \mathbf{v}_{\eps})(t)}{\bpsi}_{\mathcal{D}(A)} &- \int_\Omega \rho_{\eps}\mathbf{v}_{\eps} \otimes (\mathbf{v}_{\eps}+\tbj_\eps)): D\bpsi\, dx\\\label{eq:TimeDerivative}
  &- \int_\Omega 2\nu(\varphi_{\eps}) D\mathbf{v}_{\eps}: D\bpsi \,dx   = - \int_\Omega \varphi_{\eps}\nabla \mu_{\eps} \cdot \bpsi\,dx
\end{align}
for all $\bpsi\in \mathcal{D}(A)$ and almost every $t\in (0,T)$.
Since $\mathbb{P}_{\sigma}$ is bounded in $L^2(\Omega)^d$,  $ \mathbb{P}_{\sigma} \left(\rho_{\eps} \mathbf{v}_{\eps} \right) $ is bounded in $L^2(0, T; L^2(\Omega)^d)$. 
Moreover, $ \rho_{\eps} \mathbf{v}_{\eps} \otimes \mathbf{v}_{\eps}
$ is bounded in $ L^2(0, T; L^{\frac32}(\Omega)^{d\times d}) $ and $\mathbf{v}_{\eps} \otimes \tbj_{\eps}$ is bounded in $L^{\frac87}(0, T; L^{\frac43}(\Omega)^{d\times d})$ since $\tbj_\eps= -\tfrac{\tilde{\rho}_2-\tilde{\rho}_1}2 m(\varphi_\eps) \nabla \varphi_\eps$ is bounded in $L^2(0,T;L^2(\Omega)^d)$ and $\ve_\eps$ is bounded in $L^{\frac83}(0,T;L^4(\Omega)^d)$. Using these bounds and the boundedness of $2 \eta(\varphi_{\eps}) D\ve_{\eps}$ in $L^2(0, T; L^2(\Omega)^{d\times d})$, we have that $\partial_t \left( \mathbb{P}_{\sigma}(\rho_{\eps} \mathbf{v}_{\eps})  
\right) $ is bounded in $ L^{\frac87}(0, T; W^{-1}_{\frac43, \sigma}(\Omega)) $ because of \eqref{eq:TimeDerivative}, where $W^{-1}_{\frac43,\sigma}(\Omega)= ( W^1_{4,0}(\Omega)\cap L^2_\sigma(\Omega))'$. Since $L^2_\sigma(\Omega)$ is compactly embedded in $H^{-1}_\sigma (\Omega):= (H^1_0(\Omega)^d\cap L^2_\sigma(\Omega))'$ and $H^{-1}_\sigma (\Omega)$ is continuously embedded in $ W^{-1}_{\frac43,\sigma }(\Omega)
$, the Aubin-Lions' lemma yields that
\begin{equation*}
 \mathbb{P}_{\sigma} (\rho_{\eps} \mathbf{v}_{\eps})\to\we_1 \qquad \text{in } L^2(0, T; H^{-1}_{\sigma}(\Omega))
\end{equation*}
for  some $\we_1$ in $L^2(0, T; H^{-1}_{\sigma}(\Omega))$ and a suitable subsequence. Since $\mathbb{P}_{\sigma}(\rho(\varphi_{\eps})\mathbf{v}_{\eps}) \rightharpoonup \mathbb{P}_{\sigma}(\rho(\varphi)\mathbf{v}) $ weakly in $L^2(0, T; L^2(\Omega))$, $\we_1 = \mathbb{P}_{\sigma}(\rho(\varphi)\mathbf{v})$. 
Hence we have
\begin{equation} \label{str-conv-1}
\mathbb{P}_{\sigma}(\rho_{\eps} \mathbf{v}_{\eps}) \rightarrow \mathbb{P}_{\sigma}(\rho \mathbf{v}) \qquad \text{in } L^2(0, T; H^{-1}_{\sigma}(\Omega))
\end{equation}
Because of the boundedness of $\partial_t \left( \mathbb{P}_{\sigma}(\rho_{\eps} \mathbf{v}_{\eps})  
\right) $ in $ L^{\frac87}(0, T; W^{-1}_{\frac43, \sigma}(\Omega)) $ and $\mathbb{P}_{\sigma}(\rho(\varphi_{\eps})\mathbf{v}_{\eps}) \rightharpoonup \mathbb{P}_{\sigma}(\rho(\varphi)\mathbf{v}) $ weakly in $L^2(0, T; L^2(\Omega))$, we have
\begin{equation*}
  \partial_t \left( \mathbb{P}_{\sigma}(\rho_{\eps} \mathbf{v}_{\eps}) \right) \rightharpoonup \partial_t \left( \mathbb{P}_{\sigma}(\rho \mathbf{v}) \right)\qquad \text{ weakly in }L^{\frac87}(0, T; W^{-1}_{\frac43, \sigma}(\Omega)).
\end{equation*} 
Using also the boundedness of $\mathbf{v}_{\eps}$ in $L^2(0, T;H^1_0(\Omega)^d)$, we conclude that $\mathbf{v}_{\eps}$ converges weakly to $\mathbf{v}$  in  $L^2(0, T;H^1_0(\Omega)^d)$ for some subsequence.
Combining this with \eqref{str-conv-1}, we obtain
\begin{equation*}
\int_{Q_T} \rho_{\eps} |\mathbf{v}_{\eps}|^2 \,d(x,t) = 
\int_{Q_T} \mathbb{P}_{\sigma}(\rho_{\eps} \mathbf{v}_{\eps}) \cdot \mathbf{v}_{\eps}\, d(x,t) \to  
\int_{Q_T} \mathbb{P}_{\sigma}(\rho \mathbf{v}) \cdot \mathbf{v} \, d(x,t) = \int_{Q_T} \rho |\mathbf{v}|^2\, d(x,t).
\end{equation*}
Together with the weak convergence of $\ve_\eps$ and $\rho_{\eps}^{\frac12} \mathbf{v}_{\eps}$ in $L^2(Q_T)^d$, we conclude that $\rho_{\eps}^{\frac12} \mathbf{v}_{\eps} \to \rho^{\frac12} \mathbf{v}$ strongly in $L^2(0, T; L^2(\Omega)^d)$. Moreover, since $\rho(\varphi_{\eps}) \to \rho(\varphi)$ almost everywhere,  $\rho_{\eps}\geq c$ for some $c>0$, the convergence $\rho_{\eps}^{\frac12} \mathbf{v}_{\eps} \to \rho^{\frac12} \mathbf{v}$  in $L^2(0, T; L^2(\Omega)^d)$ implies $\mathbf{v}_{\eps} \to \mathbf{v}$ in $L^2(0, T; L^2(\Omega)^d)$, i.e., \eqref{eq:Conv4} holds true.   

Since $\mathbf{v}_{\eps}\otimes\mathbf{v}_{\eps}$ is bounded in $L^{\frac{5}{3}}(Q_T)^{d\times d}$, it converges weakly to some $\we_2$ in $L^{\frac{5}{3}}(Q_T)^{d\times d}$. On the other hand, since $\mathbf{v}_{\eps} \rightarrow \mathbf{v}$ in $L^2(Q_T)^d$, $\mathbf{v}_{\eps} \otimes \mathbf{v}_{\eps} \rightarrow \mathbf{v} \otimes \mathbf{v}$ in $L^1(Q_T)^{d\times d}$. Hence $\we_2 = \mathbf{v} \otimes \mathbf{v}$. This means $\mathbf{v}_{\eps}\otimes\mathbf{v}_{\eps}$ converges weakly to $\mathbf{v} \otimes \mathbf{v}$ in $L^{\frac{5}{3}}(Q)^{d\times d}$. Since $\rho(\varphi_{\eps})$ converges strongly to 
$\rho(\varphi)$ in $L^p(Q_T)$ for any $1\leq p<\infty$, we conclude that
$\rho(\varphi_{\eps}) \mathbf{v}_{\eps} \otimes \mathbf{v}_{\eps}$ converges weakly to $\rho(\varphi) \mathbf{v} \otimes \mathbf{v}$ in $L^{\frac{5}{3}-\gamma}(Q_T)^{d\times d}$ for any $\gamma\in (0,\tfrac53)$. %Hence $\mathrm{div}( \rho(\varphi_{\eps}) \mathbf{v}_{\eps} \otimes \mathbf{v}_{\eps})$ converges weakly to $\mathrm{div}( \rho(\varphi) \mathbf{v} \otimes \mathbf{v})$ in $L^{\frac{10}{6}-\gamma}(0, T; W^{-1, \frac{10}{6}-\gamma}(\Omega))$ for any $\gamma>0$.
Furthermore, since $\nu(\varphi_{\eps}) D\mathbf{v}_{\eps}$ is bounded in $L^2(Q_T)^{d\times d}$, it converges weakly to some $\we_3$ in $L^2(Q_T)^{d\times d}$. On the other hand, since $D\mathbf{v}_{\eps}$ converges weakly to $D\mathbf{v}$ in $L^2(Q_T)^{d\times d}$ and $\nu(\varphi_{\eps})$ converges strongly to $\nu(\varphi)$ in $L^p(Q)$ for any $1\leq p<\infty$, $\nu(\varphi_{\eps}) D\mathbf{v}_{\eps}$ converges weakly to $\nu(\varphi) D\mathbf{v}$ in $ L^{2-\gamma}(Q_T)^{d\times d} $ for 
any $\gamma\in (0,2)$. Hence $\we_3 = \nu(\varphi) D\mathbf{v}$. Similarly as above, one shows $\tbj_{\eps} = - \beta m(\varphi_{\eps}) \nabla \mu_{\eps} \rightharpoonup \tbj = - \beta m(\varphi) \nabla \mu$ in $L^2(Q_T)^d$. Using this together with $\mathbf{v}_{\eps} \to \mathbf{v}$ in $L^2(0, T; L^2(\Omega)^d)$, we have that $\mathbf{v}_{\eps}\otimes \tbj_{\eps} \rightharpoonup \mathbf{v} \otimes \tbj$ in $L^1(Q_T)^{d\times d}$.
Similarly as above, we see $\varphi_{\eps} \nabla \mu_{\eps} \rightharpoonup \varphi \nabla \mu$ in $L^2(Q)^d$. Hence we can pass to the limit in the weak form of \eqref{eq:1}.

Since $ \partial_t \varphi_{\eps} $ is bounded in $L^2(0, T;(H^1(\Omega))')$, $\partial_t \varphi_{\eps}$ converges weakly to $\partial_t \varphi$ in $L^2(0, T;(H^1(\Omega))')$.  
Since $\mathbf{v}_{\eps} \rightarrow \mathbf{v}$ strongly in $L^2(Q_T)^d$ and
$\varphi_{\eps} \to \varphi$ almost everywhere and is uniformly bounded, we have that $\mathbf{v}_{\eps} \varphi_{\eps} \to \mathbf{v} \varphi$ strongly in $L^2(Q_T)^d$ by Lebesgue's dominated convergence theorem. We also have $ m(\varphi_{\eps}) \nabla \mu_{\eps}$ converges weakly to $ m(\varphi) \nabla \mu $ in $L^2(Q)^d$. Thus we can pass to the limit in the weak form of \eqref{eq:3}. 

The following argument is from Chapter 5  in \cite{DavoliEtAlW11}. We repeat the argument for the  convenience of the reader. %We set $ \xi_{\eps} := F'_0(\varphi_{\eps})$.
Testing \eqref{eq:4} by $F'_0(\varphi_{\eps})$, taking into account that $(\mu_{\eps})_{\eps\in (0,1)}$ is bounded  in $L^2(0, T; L^2(\Omega))$ and using the monotonicity of $ F'_0 $, we derive that
\begin{align}
\| F'_0(\varphi_\eps) \|_{L^2(0, T; L^2(\Omega))} \leq M \label{est:xi}
\end{align}
for some $M>0$ independent of $\eps \in (0,1)$.
From this and \eqref{eq:4}, we have
\begin{align}
\| a_{\eps} \varphi_{\eps} - J_{\eps} * \varphi_{\eps} \|_{L^2(0, T; L^2(\Omega))} \leq M. \label{est:J}
\end{align}
Because of \eqref{est:xi} and \eqref{est:J}, there exist
$
\xi, \eta \in L^2(0, T; L^2(\Omega)) 
$
such that
\begin{alignat*}{2}
F_0'(\varphi_\eps) &\rightharpoonup \xi &\qquad &\text{in }L^2(0, T; L^2(\Omega)), \\
a_\eps \varphi_{\eps} - J_{\eps} \ast \varphi_{\eps} &\rightharpoonup \eta &&  \text{in }L^2(0, T; L^2(\Omega))
\end{alignat*}
for a suitable subsequence.
%Using $\varphi_{\eps} \rightarrow \varphi$ in $C([0, T]; L^2(\Omega))$ 
%and maximal monotonicity of $F_0'$,
%we have $ \xi = F'_0(\varphi) $. (cf. )
Using $\varphi_{\eps} \rightarrow \varphi$ in $C([0, T]; L^2(\Omega))$ one can deduce $F'_0(\varphi_\eps)\to F'_0(\varphi)$ almost everywhere, cf. e.g.~\cite[page 1093]{AbelsBosiaGrasselli}, and therefore in $L^q(Q_T)$ for every $1\leq q <2$.

Passing to the limit in the weak formulation of \eqref{eq:3} 
we have 
\begin{align*}
\duality{\partial_t \varphi(t)}{\psi(t)}_{H^1(\Omega)} + \int_{\Omega}m(\varphi(x,t)) \nabla \mu(x, t) \cdot \nabla \psi(x) dx = \int_{\Omega} \varphi(x, t) \mathbf{v}(x, t) \cdot \nabla 
\psi(x) dx
\end{align*} 
for every $\psi \in H^1(\Omega)$ and for almost every $t\in (0, T)$ and that 
$\mu = \eta + \xi - \theta_c \varphi$.
It only remains to show that $ \varphi \in L^{\infty}(0, T; H^1(\Omega)) \cap L^2(0,T; H^2(\Omega)) $ and $\eta = - \Delta \varphi$.

Because of $ \varphi_{\eps} \rightarrow \varphi $ in $C([0, T]; L^2(\Omega))$, 
Lemma~\ref{lem:3.3}, and the energy estimate, we conclude
\begin{align*}
\| E^0(\varphi) \|_{L^{\infty}(0, T)} \leq \liminf_{\eps \rightarrow 0} \| E_{\eps}^0(\varphi_{\eps}) \|_{L^{\infty}(0, T)} \leq M 
\end{align*}
Hence $ \varphi \in L^{\infty}(0, T; H^1(\Omega))$. Since $E_\eps^0(\varphi_\eps)$ is quadratic in $\varphi_\eps$, we have
\begin{align*}
\int_0^T E_{\eps}^0(\varphi_{\eps}(t))\, dt + \int_{Q_T} (a_\eps \varphi_{\eps} - J_{\eps} * \varphi_{\eps})(t, x)(\psi - \varphi_{\eps})\, dx\, dt \leq \int_0^T E^0_{\eps}(\psi(t))\, dt
\end{align*}
for any $\psi\in H^1(\Omega)$ with $(\psi)_\Omega=m$. 
Since $\varphi_{\eps} \rightarrow \varphi$ in $C([0, T]; L^2(\Omega))$, using Lemma~\ref{lem:3.3} and Fatou's lemma, we derive 
\begin{align*}
\frac{1}{2} \int_{Q_T} |\nabla \varphi(t, x)|^2 \,dx\, dt + \int_{Q_T} \eta(t, x) (\psi - \varphi)(t, x) d(x,t) \leq \frac{1}{2} \int_{Q_T} |\nabla \psi(t, x) |^2\, d(x,t)
\end{align*}
for every $\psi \in L^2(0, T; H^1(\Omega))$ with $(\psi(t))_\Omega=m$ for almost every $t\in (0,T)$. 

If we take $\psi(t, x) = \varphi(t, x) + h \chi(t) \tau(x)$, where $h \in \mathbb{R}$, $\chi \in C([0, T])$ and $\tau \in H^1_{(0)}(\Omega)$, and passing to the limit $h \rightarrow 0$,
we obtain 
\begin{align*}
\int_{\Omega} \eta(x, t) \tau(x) dx = \int_{\Omega} \nabla \varphi(x, t) \cdot \nabla \tau(x) dx
\end{align*}
for a.e.\ $t \in (0, T)$ and for all $\tau \in H^1_{(0)}(\Omega)$.  
By  classical elliptic regularity theory, we conclude that $\varphi \in L^2(0, T; H^2(\Omega))$ and $\eta = - \Delta \varphi$ and $\frac{\partial \varphi}{\partial \mathbf{n}}|_{\partial\Omega} = 0$. Furthermore, since $\varphi \in L^\infty(0,T;H^1(\Omega))\cap H^1(0,T;(H^1(\Omega))')$, %and $\ve \in L^\infty(0,T;L^2_\sigma(\Omega))\cap W^1_{\frac87}(0,T;W^{-1}_{\frac43,\sigma}(\Omega))$,
we have $\varphi \in BC_w([0,T]; H^1(\Omega))$, cf.\ e.g.\ \cite[Lemma~4.1]{LTModel}. Moreover, using the same arguments as in \cite[Section~5.2]{AbelsDepnerGarcke} one shows $\ve \in BC_w([0,T]; L^2_\sigma(\Omega))$ and $\ve|_{t=0}=\ve_0$.

Finally, we prove the energy inequality for the limit $(\mathbf{v}, \varphi, \mu)$. Using \eqref{eq:Energy}, we obtain
\begin{align}\label{eq:Energy-epsilon}
\mathcal{E}_\eps(\ve_{\eps}(t),\varphi_{\eps}(t))+ \int_{Q_t} (2\nu(\varphi_{\eps})|D\ve_{\eps}|^2+ m(\varphi_{\eps})|\nabla \mu_{\eps}|^2) \,d(x,\tau) \leq \mathcal{E}_\eps(\ve_{\eps}(0),\varphi_{\eps}(0))
\end{align}
If we take liminf of both sides of \eqref{eq:Energy-epsilon} as 
$\eps \searrow 0$, we have
\begin{align*}
\liminf_{\eps \searrow 0} \mathcal{E}_{\eps}(\ve_{\eps}(t),\varphi_{\eps}(t))+ \liminf_{\eps \searrow 0} \int_{Q_t} (2\nu(\varphi_{\eps})|D\ve_{\eps}|^2+ m(\varphi_{\eps})|\nabla \mu_{\eps}|^2) \,d(x,\tau) \leq \lim_{\eps \searrow 0} \mathcal{E}_\eps(\ve_{\eps}(0),\varphi_{\eps}(0)),
\end{align*} 
where from our assumption on the sequence of the initial data, we conclude  
\begin{align*}
\liminf_{\eps \searrow 0} \mathcal{E}_\eps(\ve_{\eps}(0),\varphi_{\eps}(0)) 
= \lim_{\eps \searrow 0} \mathcal{E}_\eps(\ve_{0,\eps},\varphi_{0,\eps})
= \mathcal{E}(\ve_0,\varphi_0).
\end{align*}
For almost all $t \in (0, T)$, we have 
\begin{align*}
 \mathcal{E}(\ve(t),\varphi(t))\leq \liminf_{\eps \searrow 0} \mathcal{E}_{\eps}(\ve_{\eps}(t),\varphi_{\eps}(t)) 
\end{align*}
because of $\ve_{\eps}(t) \rightarrow \ve(t) $ in $ L^2(\Omega)^d$, $\varphi_{\eps}(t)\to_{\eps\to 0} \varphi(t)$ in $L^2(\Omega)$ for almost every  $t\in (0, T)$, and Lemma~\ref{lem:3.3}.
Furthermore, for any $ t \in (0, T)$ we obtain
\begin{align*}
&\int_{Q_t} (2\nu(\varphi)|D\ve|^2+ m(\varphi)|\nabla \mu|^2) \,d(x,\tau) \leq \liminf_{\eps \searrow 0} \int_{Q_t} (2\nu(\varphi_{\eps})|D\ve_{\eps}|^2+ m(\varphi_{\eps})|\nabla \mu_{\eps}|^2) \,d(x,\tau)
\end{align*}
using weak lower semicontinuity of norms  
and
$\nu(\varphi_{\eps})^{\frac12} D \ve_{\eps} \rightharpoonup \nu(\varphi)^{\frac12} D \ve$ weakly in $L^2(0, T; L^2(\Omega))$ and 
$m(\varphi_{\eps})^{\frac12} \nabla \mu_{\eps} \rightharpoonup 
m(\varphi)^{\frac12} \nabla \mu$ weakly in $L^2(0, T; L^2(\Omega))$. In summary, we have shown \eqref{eq:Energy2} for almost every $t\in (0,T)$. But using $\ve\in BC_w([0,T];L^2_\sigma(\Omega))$, $\varphi \in C([0, T]; L^2(\Omega)) \cap BC_w([0,T];H^1(\Omega))$, hence $\rho^{\frac{1}{2}} \ve \in  BC_w([0, T]; L^2(\Omega))$, and suitable properties of $\mathcal{E}$ which concerns continuity or weak lower semi-continuity of each terms, we finally obtain \eqref{eq:Energy2} for every $t\in [0,T]$ by a density argument.    
This completes the proof of Theorem~\ref{main-thm}.   
\end{proof}

%%%%%%%%%%%%%%%%%%%%%%%%%%%%%%%%%%%%%%%%%%%%%%%%%%%%%%%%%%%%%%%%%%%%%%%%%%%%%%%%%%%%%%%%%%%%%%%%
%%%%%%%%%%%%%%%%%%%%%%%%%%%%%%%%%%%%%%%%%%%%%%%%%%%%%%%%%%%%%%%%%%%%%%%%%%%%%%%%%%%%%%%%%%%%%%%%
%%%%%%%%%%%%%%%%%%%%%%%%%%%%%%%%%%%%%%%%%%%%%%%%%%%%%%%%%%%%%%%%%%%%%%%%%%%%%%%%%%%%%%%%%%%%%%%%

\section*{Acknowledgments}
The second author has been supported by JSPS KAKENHI number 17K17804.
This support is gratefully acknowledged.

%\bibliographystyle{abbrv}
% \bibliographystyle{siam}

% \bibliography{Bibliography} %\nocite{*}l

\def\cprime{$'$} \def\ocirc#1{\ifmmode\setbox0=\hbox{$#1$}\dimen0=\ht0
  \advance\dimen0 by1pt\rlap{\hbox to\wd0{\hss\raise\dimen0
  \hbox{\hskip.2em$\scriptscriptstyle\circ$}\hss}}#1\else {\accent"17 #1}\fi}

\end{document}